\documentclass{birkjour}
\usepackage{amssymb,anysize,float,booktabs,helvet,paralist}
\usepackage[shortlabels]{enumitem}
\usepackage{hyperref}
\hypersetup{citecolor=red, linkcolor=blue, colorlinks=true}

\newtheorem{theorem}{Theorem}[section]
\newtheorem{lemma}[theorem]{Lemma}
\newtheorem{proposition}[theorem]{Proposition}
\newtheorem{corollary}[theorem]{Corollary}
\newtheorem{remark}[theorem]{Remark}
\newtheorem{example}[theorem]{Example}
\newtheorem{definition}[theorem]{Definition}

\renewcommand\ge{\geqslant}
\renewcommand\le{\leqslant}

\DeclareMathOperator\p{\mathcal{P}}
\DeclareMathOperator\gf{\mathrm{GF}}
\DeclareMathOperator\pg{\mathrm{PG}}
\DeclareMathOperator\h{\mathsf{H}}
\DeclareMathOperator\w{\mathsf{W}}
\DeclareMathOperator\q{\mathsf{Q}}
\DeclareMathOperator\cO{\mathcal{O}}
\DeclareMathOperator\cS{\mathcal{S}}
\DeclareMathOperator\cI{\mathcal{I}}
\DeclareMathOperator\cL{\mathcal{L}}
\DeclareMathOperator\cP{\mathcal{P}}
\DeclareMathOperator\setC{\mathbb{C}}
\DeclareMathOperator\cW{\mathcal{W}}
\DeclareMathOperator\cQ{\mathcal{Q}}
\newcommand\posev{r}
\newcommand\negev{s}

\setlist[description]{leftmargin=1ex,font=\normalfont\bfseries\space,style=nextline}

\title{New non-existence proofs for ovoids of Hermitian polar spaces and hyperbolic quadrics} 

\author{John Bamberg}
\address{ %
Centre for the Mathematics of Symmetry and Computation\\
School of Mathematics and Statistics\\
The University of Western Australia\\
35 Stirling Highway, Crawley, W.A. 6009, Australia.}
\email{John.Bamberg@uwa.edu.au}

\author{Jan De Beule}
\address{Department of Mathematics,\\
Ghent University,\\
Krijgslaan 281 - S22,\\
9000 Ghent,\\
Belgium.}
\email{jdebeule@cage.ugent.be}

\author{Ferdinand Ihringer}
\address{ %
Mathematisches Institut,\\
Justus Liebig University Giessen,\\
Arndtstra\ss{}e 2,\\
35392 Giessen,\\
Germany.}
\email{Ferdinand.Ihringer@math.uni-giessen.de}

\thanks{John Bamberg was supported by an Australian Research Council Future Fellowship FT120100036.\\ 
Jan De Beule is a postdoctoral fellow of the Research Foundation Flanders -- FWO (Belgium), and acknowledges 
the Research Foundation Flanders -- FWO (Belgium) for a travel grant.}

\begin{document}

\begin{abstract}
We provide new proofs for the non-existence of ovoids in hyperbolic spaces of rank at least four in
even characteristic, and for the Hermitian polar space $\h(5, 4)$. We also improve the results of A.
Klein on the non-existence of ovoids of Hermitian spaces and hyperbolic quadrics.
\end{abstract}

\maketitle

\section{Introduction}

Finite polar spaces (of which finite classical polar spaces are a subclass), are the natural
geometries for the finite classical simple groups. The main geometric property of polar spaces, as
shown by Buekenhout and Shult \cite{Buekenhout:1974kx}, is the so-called \emph{one or all axiom}. A
polar space is a point-line geometry\footnote{in fact, a partial linear space} with the property
that if $P$ is a point and $\ell$ a line not incident with $P$, then $P$ is collinear with either
all points of $\ell$ or with exactly one point of $\ell$. Polar spaces are some of the most
important examples in the theory of incidence geometries and provide a rich class of spherical
buildings. We will be interested in finite classical polar spaces which arise from equipping a
finite vector space with a sesquilinear or quadratic form, and they are the only finite polar spaces
of rank at least $3$ \cite{Tits:1974ys}.

According to Dembowski \cite[footnote page 48]{Dembowski:1997fj}, an \emph{ovoid} of a projective
space is defined for the first time by Tits in \cite{Tits:1962vn}, where Tits began his introduction
with reference to the works of Barlotti and Segre (\cite{Barlotti:1955yq} and
\cite{Segre:1959,Segre:1959rz}) on arcs and caps of finite projective spaces. The connection with
polar spaces is due to Tits' geometric construction of the Suzuki groups $\,^2B_2(2^{2h+1})$ (see
also \cite{Tits:1962vn}), where Tits' ovoids can also be realised in a second sense, as
\emph{ovoids} of a particular finite classical polar space (i.e., rank $2$ symplectic spaces).
Furthermore, the Ree groups $\,^2G_2(3^{2h+1})$ can also be naturally constructed from ovoids of a
certain family of finite classical polar spaces (i.e., rank $3$ parabolic quadrics). In 1974, J. A.
Thas \cite{Tits:1974ys} synthesised these objects as ovoids of finite polar spaces; a set $\cO$ of
points such that every generator meets $\cO$ in exactly one point.

To use the words of J. A. Thas \cite{Thas:2001ly}, ovoids of finite polar spaces have ``many
connections with and applications to projective planes, circle geometries, generalised polygons,
strongly regular graphs, partial geometries, semi-partial geometries, codes, designs". In
\cite{Thas:1981kx}, the existence and non-existence of ovoids of finite polar spaces was first
investigated, and a striking pattern emerges when one looks to the known examples: the existence of
ovoids of finite polar spaces simulates the results of Tits for ovoids of projective spaces. That
is, ovoids seem to exist only when the rank of the geometry\footnote{The rank of a geometry is
simply its number of types of object.} is small. Shult \cite[\S2.2]{Shult:2005yu} all but
conjectures that if a finite polar space of rank $r$ possesses an ovoid, then $r\le 4$. Thas
\cite{Thas:1981kx} showed that this property holds true for the symplectic and elliptic polar
spaces, but also for Hermitian polar spaces in even dimension. The best known bounds on ovoids of
Hermitian spaces and hyperbolic quadrics are due to Blokhuis and Moorhouse
\cite{Blokhuis1995,Moorhouse1996}. For instance, if $p$ is prime and $q=p^h$, there are no ovoids of
$\q^{+}(2r+1,q)$ when $p^{r}>\binom{2r+p}{2r+1}-\binom{2r+p-2}{2r+1}$. Thus for $p=2$ or $p=3$, we
see that no ovoids exist for $r\ge 4$. Cooperstein \cite{Cooperstein:1995aa} gave another proof that
$\q^{+}(9,q)$ has no ovoids for $q$ even by relating the existence of an ovoid in the quadric to the
maximum size of a cap in a symplectic polar space. We give a geometric proof for the non-existence
of ovoids of $\q^+(9,q)$, $q$ even (Section \ref{sec:q+9}), that utilises a direct parity argument
in the hyperbolic quadric.

For Hermitian spaces $\h(2r-1,q^2)$ in projective spaces of odd dimension, Moorhouse
\cite{Moorhouse1996} proved that no ovoids exist when
\[
p^{2r-1}>\binom{p+2r-2}{2r-1}^2-\binom{p+2r-3}{2r-1}^2.
\]
Again, this bound implies that for $r\ge 4$, $\h(2r-1,q^2)$ has no ovoids if $q\in\{2,3\}$. Using a
combinatorial argument in an inductive way, Klein \cite{Klein2004} proved that no ovoid of
$\h(2r-1,q^2)$ can exist for $r>q^3+1$, which although not as powerful as the Moorhouse bound, has a
simpler proof and provides greater insight into the existence of ovoids of Hermitian spaces. Using a
particular tight set we provide a stronger induction base for Klein's argument, that improves his
bound to $r > q^3-q^2+2$ (Section \ref{sec:AKlein}).

The only result known for the non-existence of ovoids of Hermitian polar spaces with rank less than
$4$ is the theorem of De Beule and Metsch \cite{DeBeule2005} that $\h(5,4)$ has no ovoid. We provide
a proof for the non-existence of ovoids of $\h(5,4)$ (Section \ref{sec:h54}) based on the use of
particular intriguing sets. The method employed in \cite{DeBeule2005} relies on the characterisation
of ovoids in $\h(3, 4)$, whereas our proof observes a parity argument within $\h(5, q^2)$. Hence our
proof opens the way for further investigation into the existence of ovoids of $\h(5,q^2)$ for $q>2$.

\section{Preliminary background}

Throughout, we will use the symbol $q$ for a prime power $q:=p^h$, $p$ prime and $h \ge 1$, and we
will denote the the finite field of order $q$ as $\gf(q)$. The vector space of dimension $d$ over
$\gf(q)$ will be written as $V(d,q)$, and $\pg(n,q)$ will denote the projective space with
underlying vector space $V(n+1,q)$. Let $f$ be a (reflexive) sesquilinear or quadratic form on
$V(n+1,q)$. The elements of the finite classical polar space $\p$ associated with $f$ are the
totally singular or totally isotropic subspaces of $\pg(n,q)$ with relation to $f$, according to
whether $f$ is a quadratic or sesquilinear form. The Witt index of the form $f$ determines the the
dimension of the subspaces of maximal dimension contained in $\p$; the {\em rank} $\p$ equals the
Witt index of its form, and the (projective) dimension of generators will be one less than the Witt
index. Hence, a finite classical polar space of rank $r$ embedded in $\pg(n,q)$ has an underlying
form of Witt index $r$, and contains points, lines, \ldots, $(r-1)$-dimensional subspaces. The
elements of maximal dimension are called its \emph{generators}.

We will use projective notation for finite polar spaces so that they differ from the standard
notation for their collineation groups. For example, we will use the notation $\w(d-1,q)$ to denote
the symplectic polar space coming from the vector space $V(d,q)$ equipped with a non-degenerate
alternating form. Here is a summary of the notation we will use for finite polar spaces, together
with their ovoid numbers (which we define below).

\begin{table}[H]
\caption{Notation for the finite classical polar spaces, together with their ovoid numbers.}\label{table:overview}
\begin{tabular}{lcccc}
\toprule
Polar Space&Notation&Collineation Group&Ovoid Number& Type $e$\\
\midrule
Symplectic&$\w(d-1,q)$, $d$ even&$\mathsf{P\Gamma Sp}(d,q)$&$q^{d/2}+1$ & $1$\\
Hermitian&$\h(d-1,q^2)$, $d$ odd&$\mathsf{P\Gamma U}(d,q)$&$q^{d}+1$ & $3/2$\\
Hermitian&$\h(d-1,q^2)$, $d$ even&$\mathsf{P\Gamma U}(d,q)$&$q^{d-1}+1$ & $1/2$\\
Orthogonal, elliptic&$\mathsf{Q}^-(d-1,q)$, $d$ even&$\mathsf{P\Gamma O}^-(d,q)$&$q^{d/2}+1$ & $2$\\
Orthogonal, parabolic&$\mathsf{Q}(d-1,q)$, $d$ odd&$\mathsf{P\Gamma O}(d,q)$&$q^{(d-1)/2}+1$ & $1$\\
Orthogonal, hyperbolic&$\mathsf{Q}^+(d-1,q)$, $d$ even&$\mathsf{P \Gamma O}^+(d,q)$&$q^{d/2-1}+1$ & $0$\\
\bottomrule
\end{tabular}
\end{table}

Let $\p$ be a finite polar space defined by a sesquilinear or quadratic form $f$, and let $X$ be a
point of of the ambient projective space. Then $X^\perp$ is the set of projective points whose
coordinates are orthogonal to $X$ with respect to the form $f$. Note that when $f$ is a quadratic
form, it determines a (possibly degenerate when $q$ is even), symplectic form\footnote{When $f$ is a
quadratic form, $g(v,w) := f(v+w)-f(v)-f(w)$ is an alternating form.} $g$, and two projective points
$X$ and $Y$ are orthogonal with relation to $f$ if, by definition, they are orthogonal with relation
to $g$. The set of points $X^\perp$ is a hyperplane, and when $X$ is a point of $\p$, the hyperplane
$X^\perp$ is the {\em tangent hyperplane} at $X$ to $\p$. For any set $A$ of points, $A^\perp :=
\cap_{X \in A} X^\perp$. The following result is fundamental in the theory of finite classical polar
spaces.

\begin{lemma}
Let $\p_r$ be a finite polar space of rank $r$, $r \geq 2$, and let $X$ be a point of $\p_r$. Then
the set of points $X^\perp \cap \p_r$ is a cone with vertex $X$ and base a polar space $\p_{r-1}$ of
rank $r-1$, of the same type as $\p_r$.
\end{lemma}

As a consequence, the elements of a finite polar space $\p_r$ incident with a point $X \in \p_r$
induce a polar space of rank $r-1$ of the same type, which is called the {\em quotient space}, and
which is sometimes denoted as $X^\perp /X$. Equivalently, projecting the elements from $X$ onto any
hyperplane $\pi$ not on $X$, will yield a finite polar space isomorphic with $\p_{r-1}$ embedded in
the subspace $\pi \cap X^\perp$, \cite[p. 3]{Hirschfeld1991}\footnote{If $(V,f)$ is a formed space,
and $X$ is a totally isotropic subspace of $V$, then we can equip the quotient vector space
$X^\perp/X$ with the form $f'$ defined by $f'(X+u,X+v):=f(u,v)$, which is the algebraic counterpart
of the geometric statement.}.

\begin{definition}
Suppose that $\p$ is a finite polar space. An ovoid is a set $\cO$ of points of $\p$ such that every
generator of $\p$ meets $\cO$ in exactly one point.
\end{definition}

The non-existence of ovoids in higher rank is implied by the non-existence of ovoids in low rank.
Projection from a point $X$ not in an ovoid is well-known to produce an ovoid of the quotient polar
space (see \cite[\S2]{Kantor1982}), which we reproduce below.

\begin{lemma}\label{lemma:slicing}
Let $\mathcal{O}$ be an ovoid of a finite polar space $\p_r$ of rank $r \ge 3$ and embedded in
$\pg(d,q)$, and let $X$ be a point of $\p$ not in $\mathcal{O}$. Then $\cO$ induces an ovoid $\cO_X$
in a polar space embedded in $\pg(d-2,q)$ of the same type but of rank $r-1$.
\end{lemma}
\begin{proof}
The quotient space $X^\perp/X$ is a polar space $\p_{r-1}$ of rank $r-1$. Each generator of $\p_{r}$
meets $\cO$ in exactly one point, so after projection from $X$, the induced generators of $\p_{r-1}$
meet the projection of $\cO$ in exactly one point. So $\cO$ induces an ovoid $\cO_X$ of $\p_{r-1}$.
\end{proof}

More information on the existence of ovoids of all finite polar spaces (and references for these
examples) can be found in \cite{DBKMnova}.

\section{Intriguing sets of strongly regular graphs}

This section repeats the theory of intriguing sets of finite polar spaces as considered by
\cite{Bamberg2012, Bamberg2007, Eisfeld1998a} in the more general context of strongly regular
graphs. All the results in this section are due to Delsarte \cite{Delsarte1973, Delsarte1977}. We
include them to make the paper more self-contained.

Let $\Gamma = (X, \sim)$ be a graph, where $X$ is a set of vertices, and $\sim$ is a symmetric
relation with $\sim ~\subseteq X \times X$. We say that two vertices are \emph{adjacent} if they are
in relation $\sim$. We say that $\Gamma$ is strongly regular with parameters $(n, k, \lambda, \mu)$
if each of the following holds:
\begin{compactenum}[(i)]
  \item The number of vertices is $n$.
  \item Each vertex is adjacent with $k$ vertices.
  \item Each pair of adjacent vertices is commonly adjacent to $\lambda$ vertices.
  \item Each pair of non-adjacent vertices is commonly adjacent to $\mu$ vertices.
\end{compactenum}
The \emph{adjacency matrix} $A$ of $\Gamma$ is matrix over $\setC$ indexed by the vertices $X$ with
$(A)_{xy} = 1$ if the vertex $x$ and the vertex $y$ are adjacent and $(A)_{xy} = 0$ if the vertex
$x$ and the vertex $y$ are non-adjacent. If $0 < k < n-1$, then the adjacency matrix of $A$ has
exactly three eigenvalues $k$, $\posev$, and $\negev$ \cite[Theorem 1.3.1 (i)]{Brouwer1989}. The
eigenvalue $k$ has multiplicity $1$, and the \emph{all-ones vector} $j$ is one of its eigenvectors.
The other eigenvalues satisfy
\begin{align}
  &\posev = \frac{\lambda - \mu + \sqrt{(\lambda - \mu)^2 + 4(k-\mu)}}{2},\label{eq:identitied_eplus} \\
  &\negev = \frac{\lambda - \mu - \sqrt{(\lambda - \mu)^2 + 4(k-\mu)}}{2}.\label{eq:identitied_eminus}
\end{align}
Let $V_+$ and $V_-$ be the eigenspaces corresponding to the eigenvalues $\posev$ and $\negev$,
respectively. Since $A$ is a symmetric matrix over $\setC$, the eigenspaces of $A$ yield an
orthogonal direct sum decomposition of $\setC^n$:
\[
\setC^n = \langle j \rangle \perp V_+ \perp V_-.
\]

The parameters of a strongly regular graph are constrained by a wealth of well-known equations. For
our purposes, we need the following three equations \cite[Theorem 1.3.1]{Brouwer1989}:
\begin{align}\label{eq:identitied_srg_parameters}
  n\mu &= (k-\posev)(k-\negev), \quad \negev\posev = \mu - k,  \quad k(k-\lambda -1) = (n-k-1)\mu.
\end{align}

Following the theme of \cite{Bamberg2012}, we will introduce the notion of a \emph{weighted
intriguing set} of a strongly regular graph.

\begin{definition}
Let $\chi \in \setC^n$ and $\epsilon \in \{ -, + \}$. We say that $\chi$ is a \emph{weighted
intriguing set} if $\chi \in \langle j \rangle \perp V^\epsilon$. If $\chi$ is a $0$-$1$-vector,
then we say that $\chi$ is an \emph{intriguing set}. We call a weighted intriguing set $\chi$ in
$\langle j \rangle \perp V_-$ a \emph{weighted ovoid}. We call a weighted intriguing set $\chi$ in
$\langle j \rangle \perp V_+$ a \emph{weighted tight set}. We say that a collection of intriguing
sets have the same \emph{type} if they are either all weighted ovoids, or if they are all weighted
tight sets.
\end{definition}

In the language of Delsarte, an intriguing set is a \emph{design} for the association scheme arising
from the strongly regular graph. The following result is due to Delsarte \cite{Delsarte1973,
Delsarte1977}.

\begin{lemma}\label{lemma:equivalence_intriguingset_definition}
  Let $\Gamma = (X, \sim)$ be a strongly regular graph with parameters $(n, k, \lambda, \mu)$, $0 <
  k < n-1$ and adjacency matrix $A$. Let $\posev$ and $\negev$ be the eigenvalues of $A$ different
  to $k$. Let $\chi \in \setC^n$. Then
  \begin{align*}
    (j \chi^\top)^2 k + \negev (n \chi \chi^\top - (j \chi^\top)^2) \le n \chi A \chi^\top \le (j \chi^\top)^2 k + \posev (n \chi \chi^\top - (j \chi^\top)^2).
  \end{align*}
  Equality holds on the left hand side if and only if $\chi$ is a weighted ovoid. Equality holds on
  the right hand side if and only if $\chi$ is a weighted tight set. In particular, if $\chi$ is a
  $0$-$1$-vector, then
  \begin{align*}
    (j \chi^\top)^2 k + \negev j \chi^\top (n- j \chi^\top) \le n \chi A \chi^\top \le (j \chi^\top)^2 k + \posev j \chi^\top (n- j \chi^\top).
  \end{align*}
\end{lemma}

If the vector $\chi$ is the characteristic vector of a set $Y$,\footnote{This means $\chi_x = 1$ if
$x \in Y$, and $\chi_x = 0$ if $x \in X \setminus Y$.} then $\chi A \chi^\top/j\chi^\top = \chi A
\chi^\top/|Y|$ is the average number of adjacent vertices in $Y$. A \emph{coclique} of a graph is a
set of pairwise non-adjacent matrices (i.e., $\chi A \chi^\top=0$), while a \emph{clique} is a set
of pairwise adjacent vertices (i.e., $\chi A \chi^\top= x(x-1)$). In the collinearity graph of a
finite polar space, cocliques are traditionally called \emph{ovoids} if they satisfy the first
inequality of Lemma \ref{lemma:equivalence_intriguingset_definition}, and cliques are traditionally
called \emph{tight sets} if they satisfy the second inequality of Lemma
\ref{lemma:equivalence_intriguingset_definition}. We will adopt the same terminology for strongly
regular graphs. Lemma \ref{lemma:equivalence_intriguingset_definition} makes it clear that the
characteristic vector $\chi$ of an ovoid is a weighted ovoid with $j\chi^\top =
\frac{n\negev}{\negev-k}$, and that the characteristic vector of a tight set is a weighted tight
set\footnote{This equality is not obvious, since Lemma
\ref{lemma:equivalence_intriguingset_definition} only shows $j\chi^\top =
\frac{n(\posev+1)}{n-k+\posev}$. One can use the identities (\ref{eq:identitied_srg_parameters}) to
prove the claim.} with $j\chi^\top = 1 - \frac{k}{\negev}$. Consistent with this definition, we call
a weighted ovoid $\chi$ a \emph{weighted $m$-ovoid} if $j\chi^\top = m \frac{n\negev}{\negev-k}$,
and we call a weighed tight set $\psi$ a \emph{weighted $i$-tight set} if $j\psi^\top = i(1 -
\frac{k}{\negev})$. If a weighted $m$-ovoid $\chi$ is a $0$-$1$-vector, then we say that $\chi$ is
an $m$-ovoid. If a weighted $i$-tight set $\psi$ is a $0$-$1$-vector, then we say that $\psi$ is an
$i$-tight set. We identify $0$-$1$-vectors with the corresponding sets of vertices. Then this is
consistent with the usual definitions of $m$-ovoids and $i$-tight sets which can be found
\cite{Bamberg2007}.

The following result is crucial for the investigation of intriguing sets.

\begin{lemma}\label{lemma:intersection_tight_set_ovoid}
 Let $\chi$ be a weighted $m$-ovoid. Let $\psi$ be a weighted $i$-tight set. Then $\chi\psi^\top =
 mi$.
\end{lemma}
 \begin{proof}
   By definition, there exist $\chi_- \in V_-$ and $\psi_+ \in V_+$ such that $\chi = m
   \frac{\negev}{\negev-k} j + \chi_-$ and $\psi = i(1 - \frac{k}{\negev}) j/n + \psi_+$. The
   vectors $\chi$ and $\psi$ are orthogonal, hence $ \chi\psi^\top = mi \frac{\negev}{\negev-k} (1 -
   \frac{k}{\negev}) j j^\top = mi$.
 \end{proof}

Let $G$ be a group of automorphisms of the graph $\Gamma$. We say that $G$ {\em acts generously
transitive} on $\Gamma$, if for all vertices $x,y \in \Gamma$, there exists an automorphism $g \in
G$ such that $x^g = y$ and $y^g = x$. We say that a weighted intriguing set is \emph{non-trivial} if
it is not in the span of $j$. The following characterisation of intriguing sets, which follows
immediately from the definitions, is very helpful.

\begin{proposition}\label{propn:characterisation_by_intersection}
  Let $G$ be a group which acts generously transitively on $\Gamma$. Let $\chi, \psi \in \setC^n$
  with $\chi, \psi \notin \langle j \rangle$. Then the following statements are equivalent.
  \begin{compactenum}[(a)]
   \item One of the vectors $\chi$ and $\psi$ is a non-trivial weighted $m$-ovoid, and one of the
     vectors is a non-trivial weighted $i$-tight set.
   \item For all $g \in G$, we have $\psi^g \chi^\top = mi$.
  \end{compactenum}
\end{proposition}

\begin{corollary}\label{corollary:span_tight_sets}
  Let $G$ be a group which acts generously transitively on $\Gamma$.
  Let $\chi \in \setC^n$.
  \begin{compactenum}[(a)]
    \item If $\chi$ is a non-trivial weighted $m$-ovoid with $m \neq 0$, then $\langle j \rangle \perp V_- = \langle \chi^g: g \in G \rangle$.
    \item If $\chi$ is a non-trivial weighted $i$-tight set with $i \neq 0$, then $\langle j \rangle \perp V_+ = \langle \chi^g: g \in G \rangle$.
    \item If $\chi$ is a non-trivial weighted $0$-ovoid, then $V_- = \langle \chi^g: g \in G \rangle$.
    \item If $\chi$ is a non-trivial weighted $0$-tight set, then $V_+ = \langle \chi^g: g \in G \rangle$.
  \end{compactenum}
\end{corollary}
\begin{proof}
  Let $\chi$ be a weighted $0$-ovoid.
  Obviously, $\langle \chi^g: g \in G \rangle \subseteq V_-$. 
  Suppose contrary to our claim that $\langle \chi^g: g \in G \rangle \neq V_-$.
  Then there exists a weighted $0$-ovoid $\chi' \in V_- \setminus \langle \chi^g: g \in G \rangle$ orthogonal to $\langle \chi^g: g \in G \rangle$.
  Hence by Proposition \ref{propn:characterisation_by_intersection}, $\chi'$ is a weighted tight set; a contradiction.  
  The other cases follow similarly.
\end{proof}

Proposition \ref{propn:characterisation_by_intersection} is an excellent tool for showing that a set
is intriguing: If one knows a weighted ovoid $\chi$, then one can see if a vector $\psi$ is a
weighted tight set just by considering $\chi^g \psi^\top$. If one knows a weighted tight set $\psi$,
then one can see if a vector $\chi$ is a weighted ovoid set just by considering $\chi (\psi^g)^\top$
(for each $g\in G$). One needs some general examples for non-trivial intriguing sets to do so, and
we shall construct some simple examples in the following. For a point $x \in X$ we write $x^\sim$
for the set of all vertices adjacent with $x$. We write $\chi_M$ for the characteristic vector of a
set of vertices $M \subseteq X$.

\begin{example}\label{example:std_ovoid}
  Let $\Gamma = (X, \sim)$ be a strongly regular graph with parameters $(n, k, \lambda, \mu)$. Let
  $\alpha = \frac{k(n-k+\negev)}{n\negev + k - \negev}$. Let $x$ be a vertex of $\Gamma$. Then $
  \chi = \alpha \chi_{\{x\}} + \chi_{x^\sim}$ is a non-trivial weighted $\frac{(\alpha +
  k)(1-\frac{k}{\negev})}{n}$-ovoid.
\end{example}

\begin{proof}
  We have to show that $\chi$ satisfies the lower bound on $\chi A \chi^\top$ from Lemma
  \ref{lemma:equivalence_intriguingset_definition} with equality. By the definition of $\chi$, we
  have
  \begin{align*}
    &j\chi^\top = \alpha + k, &&\chi \chi^\top = \alpha^2 + k, &&\chi A \chi^\top = k(2\alpha + \lambda).
  \end{align*}
  Together with the equations \eqref{eq:identitied_eminus} and \eqref{eq:identitied_srg_parameters}
  one can easily check that $\chi$ reaches the lower bound in Lemma
  \ref{lemma:equivalence_intriguingset_definition}. Hence, $\chi$ is a weighted ovoid. The weight of
  the weighted ovoid follows from the definition.
\end{proof}

Similarly, we get the following analogous result for weighted tight sets.

\begin{example}\label{example:std_tightset}
  Let $\Gamma = (X, \sim)$ be a strongly regular graph with parameters $(n, k, \lambda, \mu)$. Let
  $\alpha = \frac{k(n-k+\posev)}{n\posev + k - \posev}$. Let $x$ be a vertex of $\Gamma$. Then $
  \chi = \alpha \chi_{\{x\}} + \chi_{x^\sim}$ is a non-trivial weighted $\frac{\alpha +
  k}{1-\frac{k}{\negev}}$-tight set.
\end{example}

By definition, linear combinations of weighted intriguing sets of the same type are again weighted
intriguing sets. This motivates the following lemma, whose proof is straightforward and we leave it
(as a simple exercise) for the reader.

\begin{lemma}\label{lemma:intriguingsets_from_subgroups}
  Let $G$ be a group which acts generously transitively on $\Gamma$. Let $U$ be a subgroup of $G$.
  Let $\chi$ be a non-trivial weighted intriguing set of $\Gamma$. Define $\cI$ as the set
  \begin{align*}
    \left\{ \sum_{u \in U} (\chi^g)^u: g \in G \right\}.
  \end{align*}
  Then the following holds.
  \begin{compactenum}[(a)]
    \item The elements of $\cI$ are intriguing sets of the same type as $\chi$.
    \item Let $\psi$ be a weighted intriguing set of the same type as $\chi$, fixed by $U$ (i.e.,
      $\psi^u = \psi$ for all $u \in U$). Then $\psi \in \langle \cI \rangle$.
  \end{compactenum}
\end{lemma}

Recall from Table~\ref{table:overview} that all finite polar spaces have a {\em type}, and the type
$e$ can be easily discerned from the \emph{rank $1$} example, having $q^e+1$ singular points, where
$q$ is the order of the defining field. So for example, the hyperbolic quadrics have $e=0$ since
$\q^+(1,q)$ consists of two singular points (on the projective line), yet the Hermitian spaces of
even dimension have type $e=3/2$ as $\h(2,q^2)$ has $q^3+1$ singular points. In general, $e \in \{
0, 1/2, 1, 3/2, 2\}$. This number will be useful for summarising the essential combinatorial
information of all finite polar spaces in this section. The collinearity graph of the point set of a
finite polar space of rank $d$ and type $e \in \{ 0, 1/2, 1, 3/2, 2\}$ over $\gf(q)$ is a strongly
regular graph with the following parameters \cite[Section 4]{Bamberg2007}:
\begin{center}
\begin{tabular}{l|l}
\toprule
$n = \frac{(q^{d-1+e} + 1)(q^d-1)}{q-1}$ &   $k = q \frac{(q^{d-2+e} + 1)(q^{d-1}-1)}{q-1}$\\
$\lambda = q^{d-1} - 1 + q \frac{(q^{d-2+e} + 1)(q^{d-2}-1)}{q-1}$ & $\mu = \frac{(q^{d-2+e} + 1)(q^{d-1}-1)}{q-1}$\\   
$\posev = q^{d-1}-1$& $\negev = -1 - q^{d-2+e}$ \\
 $m_r = q^e \frac{(q^d-1)(q^{d-2+e}+1)}{(q-1)(q^{e-1}+1)}$ &$m_s =q \frac{(q^{d-1}-1)(q^{d-1+e}+1)}{(q-1)(q^{e-1}+1)}$\\
\bottomrule
\end{tabular}
\end{center}

Let $P$ be a point of the finite polar space. We write $P^\sim$ for all points collinear, but not
equal to $P$. Then the weighted intriguing set in Example \ref{example:std_ovoid} equals
\begin{align*}
  -(q^{d-1} - 1) \chi_P + \chi_{P^\sim}
\end{align*}
and is a non-trivial weighted $\frac{q^{d-1}-1}{q-1}$-ovoid. The map in Example \ref{example:std_tightset} equals
\begin{align*}
  (q^{d-2+e}+1) \chi_P + \chi_{P^\sim}
\end{align*}
and is a non-trivial weighted $(q^{d-2+e}+1)$-tight set. Let $T$ be the set of points of a generator
of the considered polar space. Then $\chi_T$ is a well-known $1$-tight set \cite[\S3]{Bamberg2007}.
We will only use these three intriguing sets to show that other vectors are intriguing sets. The
collineation group of a finite polar space acts generously transitively on the point set of the
polar space, so we can use Proposition \ref{propn:characterisation_by_intersection} and Lemma
\ref{lemma:intriguingsets_from_subgroups}.

Here are some more useful weighted intriguing sets which exist in all finite polar spaces.

\begin{lemma}\label{lemma:singular_line_tightset}
  Let $0 \le s < d - 1$. Let $S$ be a totally isotropic $s$-dimensional subspace of a finite polar
  space of rank $d$ and type $e$ over a finite field with $q$ elements. Then
  \begin{align*}
    (q^{d+e-2-s}+1) \chi_S + \chi_{S^\sim}
  \end{align*}
  is a weighted $(q^{d+e-2-s}+1)$-tight set.
\end{lemma}
\begin{proof}
  Recall that a generator is a $1$-tight set and that a sum of tight sets is again a tight set.
  There are $    \prod_{i=0}^{d-2-s} (q^{i+e}+1)$
  generators through $S$, and there are $    \prod_{i=0}^{d-3-s} (q^{i+e}+1)$
  generators through $\langle S, P \rangle$ when $P \in S^\sim$.
  Hence,
  \begin{align*}
    \left(\prod_{i=0}^{d-2-s} (q^{i+e}+1) \right) \chi_S + \left(\prod_{i=0}^{d-3-s} (q^{i+e}+1) \right) \chi_{S^\sim}
  \end{align*}
  is a weighted $\left(\prod_{i=0}^{d-2-s} (q^{i+e}+1)\right)$-tight set.
  Dividing through by $\prod_{i=0}^{d-3-s} (q^{i+e}+1)$ yields our claim.
\end{proof}

\section{The non-existence proof of ovoids of $\q^+(9,q)$, $q$ even}\label{sec:q+9}

We begin with a simple property about the parabolic quadric $\q(6,q)$ that will be useful in the
proof of Theorem \ref{thm:noovoidsQ+}.

\begin{lemma}\label{le:threelines}
Let $\pi_1, \pi_2, \pi_3$ be disjoint planes of $\q(6, q)$. Then there exist exactly $q+1$ lines
meeting $\pi_1, \pi_2, \pi_3$. These lines span a subspace of dimension $3$ meeting $\q(6,q)$ in a
hyperbolic quadric $\q^+(3,q)$.
\end{lemma}

\begin{proof}
The planes $\pi_1$ and $\pi_2$ span a $5$-dimensional subspace $\alpha$ meeting $\q(6,q)$ in a
hyperbolic quadric $\q^+(5,q)$. Clearly $\alpha \cap \pi_3$ is a line $l$. A line meeting $\pi_1$,
$\pi_2$, and $\pi_3$ in a point, meets $l$ in exactly one point. Hence all these lines are contained
in the $4$-dimensional space $\langle l,\pi_2\rangle$, which meets $\pi_1$ in a line $m$. It is not
possible that two such lines $g, h$ meet $l$ in the same point $p$, since then the plane $\langle g,
h \rangle$ meets both $\pi_2$ and $\pi_3$ in a line, a contradiction since $\pi_1$ and $\pi_2$ are
disjoint. Consequently, the number of such lines is exactly $q+1$, and they span a $3$-dimensional
subspace, necessarily meeting $\q(6,q)$ in a hyperbolic quadric $\q^+(3,q)$.
\end{proof}

Let $\cO$ be an ovoid of $\q^+(9, q)$ where $q$ is even. Let $P_1, P_2 \in \cO$. Denote the section
$P_1^\bot \cap P_2^\bot \cap \q^+(9,q)$ by $\q^+_7$; whose intersection produces a quadric
isomorphic to the hyperbolic quadric $\q^+(7,q)$. Let $\sigma_1$ be a generator of $\q^+_7$ (hence
$\sigma_1$ is a solid). Consider an elliptic quadric $\q^-(3,q)$ embedded in $\sigma_1$, denote the
point set of this quadric as $\q^-_3$. Let $\Pi$ be the set of $q^2+1$ tangent hyperplanes at the
points of $\q^-_3$ in $\sigma_1$. So $\Pi$ is a set of planes. Let $\sigma_2 \subseteq \q^+_7$ be a
generator of $\q^+_7$ disjoint from $\sigma_1$. Note that for any plane $\pi \in \Pi$, $\pi^\perp$
is a $6$-dimensional space, meeting $\q^+_7$ in a cone with vertex $\pi$ and base a hyperbolic line
isomorphic to $\q^+(1,q)$. Furthermore, $\pi^\perp$ meets $\sigma_2$ in a point. Define
\begin{align*}
  \cL = \{ \langle \pi^\bot \cap \sigma_2, \pi \cap \q^-_3 \rangle: \pi \in \Pi \}.
\end{align*}
So $\cL$ is a set of $q^2+1$ lines.

\begin{lemma}\label{le:lineperps}
Let $l_1, l_2 \in \cL$ with $l_1 \neq l_2$. Let $P \in l_1$ and $Q \in l_2$ be collinear points with
$P, Q \notin \sigma_1$. Then $P \notin l_2^\perp$.
\end{lemma}
\begin{proof}
Suppose, for a proof by contradiction, that $P \in l_2^\perp$, or equivalently, $l_2 \subseteq
P^\perp$. Then $l_2 \cap \sigma_1 \subseteq \langle l_1 \cap \sigma_1, Q \rangle^\perp$. Since $P
\notin \sigma_1$, we have $\langle l_1 \cap \sigma_1, P \rangle = l_1$. Hence, $l_2 \cap \sigma_1
\subseteq l_1^\perp$. This contradicts the definition of $\cL$.
\end{proof}

\begin{lemma}\label{le:thirdgenerator1}
Let $l_1, l_2, l_3 \in \cL$ pairwise different. Let $P_1 \in l_1 \setminus (\sigma_1 \cup
\sigma_2)$. Define $P_2 = P_1^\perp \cap l_2$ and $P_3 = P_1^\perp \cap l_3$. Then $P_2$ and $P_3$
are collinear.
\end{lemma}

\begin{proof}
Let $P_1 \in l_1 \setminus (\sigma_1 \cup \sigma_2)$. Let $P_2 = P_1^\perp \cap l_2$ and $P_3 =
P_1^\perp \cap l_3$. Suppose that $P_2$ and $P_3$ are non-collinear. Set $P_1^0 = P_1$, $P_1^i =
(P_3^{i-1})^\perp \cap l_1$ for $i > 0$, $P_2^i = (P_1^i)^\perp \cap l_2$ for $i \ge 0$, and $P_3^i
= (P_2^i)^\perp \cap l_3$ for $i \ge 0$. By Lemma \ref{le:lineperps}, all $P_k^i$ are points. By
$P_2$ and $P_3$ non-collinear, $P_3 \neq P_3^0$. Let $i$ be the smallest number with $P_k^i = P_k^j$
for some $k$ and $j < i$. Then we have ${P^i_1}^\perp \supseteq \langle P^{i-1}_3, P^{j-1}_3 \rangle
= l_3$ if $k=1$ (and similar for $k=2,3$) which contradicts Lemma \ref{le:lineperps}. Hence, $P_k^i
\neq P_k^j$ for all $i \neq j$. This is a contradiction as the number of points on $l_k$ is finite.
\end{proof}

\begin{theorem}\label{thm:noovoidsQ+}
Let $q$ be even and $n\ge 4$. Then $\q^+(2n+1,q)$ has no ovoids.
\end{theorem}

\begin{proof}
By Lemma \ref{lemma:slicing}, it suffices to consider the case $n=4$. Let $\cO$ be an ovoid of
$\q^+(9,q)$. Consider two points $P_1,P_2 \in \cO$, and consider the set $\cL$. Define $\p$ as the
set of points $X \in \q^+(9, q) \setminus (P_1^\bot \cup P_2^\bot)$ such that $\langle X,l\rangle
\subseteq \q^+(9,q)$ for at least one line $l \in \cL$. We show first that for each point $X \in
\p$, there are exactly two lines $l_i \in \cL$ such that $l_i \in X^\perp$. Let $P \in \p$. Then
$P^\perp \cap P_1^\perp \cap P_2^\perp \cap \q^+(9,q)$ is a parabolic quadric section $\q_6$, and
the hyperplane $P^\perp$ meets $\q^-_3$ in $1$ or $q+1$ points.

Assume that $|P^\perp \cap \q^-_3|=1$. Then there exists exactly one line $l \in \cL$ such that $l
\in P^\perp$. Recall that $l^\perp \cap \sigma_1$ is a tangent plane of $\q^-_3$. Then $|P^\perp
\cap \q^-_3| = 1$ implies $P^\perp \cap \sigma_1 = l^\perp \cap \sigma_1$. Hence $P^\perp \cap
\q^+_7$ contains the $3$-space $\langle l, P^\perp \cap \sigma_1 \rangle$, a contradiction since
$\p^\perp \cap \q^+_7$ is a parabolic quadric $\q(6,q)$.

Before we proceed to the next case, we will show that there exists a third generator $\sigma_3$ of
$\q^+_7$ disjoint from $\sigma_1$ and $\sigma_2$, which meets all lines of $\cL$. Let $l \in \cL$
and let $P \in l \setminus \{ \sigma_1, \sigma_2 \}$. By Lemma \ref{le:thirdgenerator1}, the points
$\cP = \{ P \} \cup \{ P^\perp \cap l': l' \in \cL \setminus \{ l \}\}$ are pairwise collinear.
Hence, there exists a generator $\sigma_3$ which contains $\cP$.

Now assume that $|P^\perp \cap \q^-_3|=q+1$. Consider the three planes $P^\perp \cap \sigma_1$,
$P^\perp \cap \sigma_2$, and $P^\perp \cap \sigma_3$, which we also note are three planes of $\q_6 =
P^\perp \cap \q^+_7$. There are exactly $q+1$ lines meeting these three planes in a point by
Lemma~\ref{le:threelines}, and these lines span a $3$-space $\alpha$ meeting $\q_6$ in a hyperbolic
quadric $\q^+(3,q)$. Since $\alpha \cap \sigma_1 \cap P^\perp$ is a line, meeting $\q^-_3$ in $0,1$
or $2$ points, if a line $l \in \cL$ is contained in $P^\perp$, then $\alpha \cap \sigma_1 \cap
P^\perp$ meets $\q^-_3$. Hence, if $P^\perp$ contains exactly one line $l \in \cL$, $l' := P^\perp
\cap l^\perp \cap \sigma_1$ is a tangent line to the conic $\q^-_3 \cap P^\perp$. Consider $r \in
(\sigma_1 \cap l^\perp) \setminus P^\perp$ Then $\langle l',l \rangle$ is a plane contained in
$R^\perp / R$, a quotient geometry isomorphic with a parabolic quadric $\q(4,q)$, a contradiction.
Hence, for each point $P \in \p$, there are exactly two lines $l_i \in \cL$ such that $l_i \in
P^\perp$.

Count now the pairs $\{(l,P): l \in \cL, P \in \p \cap \cO\}$. There are $k := |\cO \cap \p|$
choices for $P$. By the above argument, there are $2$ choices for $l$ given a point $P \in \p$, so
we find $2k$ pairs. On the other hand, there are $q^2+1$ choices for a line $l \in \cL$. For each
line $l \in \cL$, $l^\perp /l$ is a quotient geometry isomorphic to $\q^+(5,q)$, so $l^\perp$
contains exactly $q^2+1$ points of $\cO$, always including $P_1,P_2$. Hence there are
$(q^2+1)(q^2-1) = q^4-1$ pairs, so $2k=q^4-1$. With $k$ a natural number, this is a contradiction
when $q$ is even.
\end{proof}

We conclude this part with two short remarks.

\begin{remark}
  We have to clearify the connection of the given proof with tight sets. Consider the situation of
  the previous proof, and the weighted tight set that consists of all generators that contain a line
  of $\cL$, but not $P_1$ or $P_2$. There are $2(q+1)(q^2+1)(q^2-1)$ such generators. Hence, we
  obtain a weighted $2(q+1)(q^2+1)(q^2-1)$-tight set as every line of $\cL$ is contained in
  $2(q+1)(q^2-1)$ such generators. By the arguments of the previous proof, we have weight $4(q+1)$
  on the points of $\cP$, but the other points of the tight set have either weight $0$ or are in
  $P_1^\perp \cup P_2^\perp$. Hence,
  \begin{align*}
    4(q+1)|\cP \cap \cO| = 2(q+1)(q^2+1)(q^2-1).
  \end{align*}
  The proof is based on a simplification of this tight set.
\end{remark}

\begin{remark}
  For $q$ even, there exists a natural embedding $\q^-(3, q) \subseteq \w(3, q) \subseteq \h(3,
  q^2)$. By field reduction one can map $\h(3, q^2)$ onto $\q^+(7, q)$. So $\w(3, q)$ is mapped onto
  the Segre variety $\cS_{3,1} = \w(3, q) \otimes \w(1, q)$. The image of $\q^-(3, q)$ gives us the
  $q^2+1$ lines $\cL \subseteq \cS_{3, 1}$.
\end{remark}

\section{The non-existence proof of ovoids of $\h(5,4)$}\label{sec:h54}

Recall that an ovoid of $\h(5, q^2)$ has size $q^5+1$. A subspace $\pi$ of $\pg(5,q)$ meets the
Hermitian variety $\h(5,q^2)$ in a set of points of a (possibly) degenerate Hermitian variety, in
other words, the underlying Hermitian form induces a (possibly) degenerate Hermitian form on the
subspace $\pi$. Such a degenerate Hermitian variety is a cone with vertex a subspace $\alpha$ of
$\pi$ and base a non-degenerate Hermitian variety in the complement of $\alpha$ in $\pi$. When we
say, for example, that a line $l$ is \emph{isomorphic} to $\h(1,q^2)$, it means that the line meets
$\h(5,q^2)$ in the point set of the Hermitian variety $\h(1,q^2)$. Likewise, when we say, a plane
$\pi$ is isomorphic to $p\h(1,q^2)$, we mean that $\pi$ meets $\h(5,q^2)$ in the cone $p \h(1,q^2)$,
$p \in l$ and $\h(1,q^2) \in p^\perp$. The following lemma could be a useful observation in the
study of ovoids of Hermitian polar spaces, as it lead to insight into the case $q=2$ of this paper.

\begin{lemma}\label{lemma:line_with_two_pts}
  Let $\cO$ be an ovoid of $\h(5, q^2)$. Then there exists a line $\ell$ isomorphic to $\h(1, q^2)$
  with $2 \le |\ell \cap \cO| < q+1$.
\end{lemma}
\begin{proof}
  Assume that all lines isomorphic to $\h(1, q^2)$ meet $\cO$ in $0$, $1$, or $q+1$ points. Call a
  line $\ell$ a \emph{Baer subline} of $\cO$ if $|\ell \cap \cO| = q+1$.

\begin{description}
\item[Case 1. All Baer sublines of $\cO$ are contained in a plane] This plane is necessarily
  isomorphic to $\h(2, q^2)$, however, $|\h(2, q^2)| = q^3+1 < |\cO|$; a contradiction.

\item[Case 2. Two Baer sublines $\ell, \ell'$ of $\cO$ span a solid] Consider the following set of
  $(q+1)^2$ lines, each isomorphic to $\h(1, q^2)$:
  \begin{align*}
    \cL = \{ PQ: P \in \ell \cap \cO, Q \in \ell' \cap \cO\}.
  \end{align*}
  Let $s \in \cL$. Since $s \cap \ell \in \cO$ and $s \cap \ell' \in \cO$, we have $|s \cap \cO| \ge
  2$. Hence, $|s \cap \cO \setminus (\ell \cup \ell')| = q-1$. Let $s' \in \cL$ with $s' \neq s$, $s
  \cap \ell \neq s' \cap \ell$, and $s \cap \ell' \neq s' \cap \ell'$. If $s \cap s'$ is not empty,
  then $\ell, \ell' \subseteq \langle s \cap \ell, s' \cap \ell, s \cap \ell', s' \cap \ell'
  \rangle$ is a plane. This contradicts our assumption that $\ell$ and $\ell'$ span a solid. Hence,
  \begin{align*}
   \left| \bigcup_{s \in \cL} s \cap \cO \right|  &= |\ell \cap \cO| + |\ell' \cap \cO| + \bigcup_{s \in \cL} | s \cap \cO \setminus (\ell \cup \ell')|\\
    &= 2(q+1) + (q+1)^2 (q-1)\\
    &= q^3+q^2+q+1.
  \end{align*}
  The subspace $S = \langle \ell, \ell' \rangle$ is isomorphic to $\h(3, q^2)$. Hence, $S^\perp$ is
  isomorphic to $\h(1, q^2)$ and contains a point $P \in \h(5, q^2) \setminus \cO$. Then $|P^\perp
  \cap \cO| \ge q^3+q^2+q+1$. This contradicts Lemma \ref{lemma:singular_line_tightset}, that
  implies $|P^\perp \cap \cO| = q^3+1$.

\item[Case 3. All Baer sublines of $\cO$ meet in a point $x$ of $\pg(5, q^2)$ and are not
  contained in a plane] Let $\ell, \ell', \ell''$ be Baer sublines of $\cO$ that span a $3$-space.
  Choose $y \in \ell$, $z \in \ell'$ with $x \neq y, z$. Then $yz$ is isomorphic to $\h(1, q^2)$
  with $|yz| \ge 2$. Hence, $yz$ is a Baer subline of $\cO$ with $x \notin yz$; a contradiction.
\end{description}
\end{proof}

\begin{lemma}[{\cite[Theorem 8]{Bamberg2007}}]\label{lemma:tightset_w5}
Let $r>1$ and let $W_{r}$ be a subgeometry of $\h(2r-1,q^2)$ isomorphic to $\w(2r-1,q)$. Then the
set of points of $W_r$ is a $(q+1)$-tight set of $\h(2r-1,q^2)$.
\end{lemma}

\begin{lemma}\label{lemma:degenerateplane}
  Let $\cO$ be an ovoid of $\h(5, 4)$. Then there exist $P, Q, R \in \cO$ such that $\langle P, Q, R
  \rangle$ is isomorphic to the degenerate Hermitian space $p\h(1, q^2)$.
\end{lemma}

\begin{proof}
  By Lemma \ref{lemma:line_with_two_pts}, there exists a line $\ell$ isomorphic to $\h(1, 4)$ with
  $|\ell \cap \cO| = 2$. Let $W_5$ be a subgeometry of $\h(5, 4)$ isomorphic to $\w(5, 2)$ with
  $\ell \subseteq W_5$. By Lemma \ref{lemma:tightset_w5}, $|W_5 \cap \cO| = 3$. Hence, $\pi =
  \langle W_5 \cap \cO \rangle$ is a plane. The plane $\pi$ is isomorphic to $p\h(1, 4)$ as $\w(5,
  2)$ does not contain non-degenerate planes.
\end{proof}

For the rest of this section, define the following objects.
\begin{enumerate}[-]
\item Let $P$ be a point of $\h(5, 4)$.
\item Let $\ell \subseteq P^\perp$ be a line of $\h(5, 4)$ isomorphic to $\h(1, 4)$.
\item Let $\cW_2$ be the set of $(q^2-1)/(q-1) = 3$ subgeometries $W_2$ of $P\ell$ isomorphic to
  $p\w(1, 2)$ with $\ell \cap H(5, 4) \subseteq W_2$.
\item Let $W_5$ be a subgeometry of $H(5, 4)$ isomorphic to $W(5, 2)$ with $W^0_2 \subseteq W_5$
  for one fixed $W_2^0 \in \cW_2$.
\item Define $W_3 = W_5 \cap \ell^\perp$.
\item Fix a subgeometry $Q^-_3$ isomorphic to $\q^-(3, 2)$ in $W_3$ with $P \in Q^-_3$.
\item Let $U$ be the intersection of the setwise stabiliser of $W_5$, the element-wise stabiliser
  of $\cW_2$, and the setwise stabiliser of $Q^-_3$. The group $U$ has size $144$.
\end{enumerate}

Let $W_2 \in \cW_2$ with $W_2 \notin W_5$. Let $\cQ^-_3$ be $\{ Q^-_3, W_3 \setminus (P^\perp \cup
Q^-_3) \}$; a partition of the points of $W_3 \setminus P^\perp$. Let $S \in \cQ^-_3$ and let
$O_{W_2, S}$ be the set of points $Q$ of $\h(5, 4)$ with
\begin{enumerate}[-]
 \item $Q \notin W_5$, $Q \notin P^\perp$,
 \item $Q^\perp \cap \ell$ is a singular point $R$ of $\h(5, 4)$,
 \item $|QR \cap S| = 1$,
 \item $Q^\perp \cap P\ell \cap W_2 \cong \w(1, 2)$.
\end{enumerate}

Let $\cO$ be an ovoid of $\h(5, 4)$. By Lemma \ref{lemma:degenerateplane}, we suppose without loss
of generality $|W_0 \cap \cO| = q+1$ and $|\ell \cap \cO| = 0$. By Lemma \ref{lemma:tightset_w5},
$|W_5 \cap \cO \setminus P\ell| = 0$. Under this assumption, the equations defined by Lemma
\ref{lemma:intriguingsets_from_subgroups} with $U$ as the group in question, imply
\begin{align*}
  |O_{W_2, S} \cap \cO| = \frac{3}{2}.
\end{align*}
This is clearly a contradiction. So by Lemma \ref{lemma:slicing}, we have.

\begin{theorem}
  The Hermitian polar spaces $\h(2n-1, 4)$ do not possess ovoids for all $n\ge 3$.
\end{theorem}

We provide a coordinatised description of $U$ and the special weighted tight set above in the
appendix. It is also possible to provide a geometrical description of the involved weighted tight
sets without too much effort. As this argument heavily relies on Lemma \ref{lemma:degenerateplane},
so $q=2$, we see no point in doing so. We do hope that it will be possible to generalise the given
construction in the future.

\section{The non-existence results of Andreas Klein and some improvements}\label{sec:AKlein}

Andreas Klein showed in \cite{Klein2004} the non-existence of ovoids in $\h(2d-1, q^2)$ if $d >
q^3+1$. This result shows the non-existence of ovoids in certain cases where the result of Moorhouse
in \cite{Moorhouse1996} does not show it and vice versa. The approach in \cite{Moorhouse1996} is
based on the computation of the $p$-rank of a generator matrix associated to a hypothetical ovoid,
while the approach in \cite{Klein2004} is purely combinatorial. With the approach followed in this
paper we can now improve Klein's result.

\begin{proposition}\label{propn:h2tightset}
  Let $d > 2$. Let $H_s$ be a subgeometry of $\h(2d-1, q^2)$ isomorphic to $\h(s, q^2)$, $1 < s <
  2d-2$, $s$ even. Then
  \begin{align*}
    (1-q^{2d-2-s}) \chi_{H_s} + (1-q^{s}) \chi_{H_s^\perp} + \chi_{C(H_s)},
  \end{align*}
  is a weighted $(q^s-1)(q^{2d-2-s}-1)$-tight set.
\end{proposition}

\begin{proof}
  By Lemma \ref{lemma:intersection_tight_set_ovoid}, Proposition
  \ref{propn:characterisation_by_intersection}, and Lemma \ref{example:std_ovoid}, we only have to
  show
  \begin{align}
    ((1-q^{2d-2}) \chi_{\{P\}} + \chi_{P^\sim}) ((1-q^{2d-2-s}) \chi_{H_s} + (1-q^{s}) \chi_{H_s^\perp} + \chi_{C(H_s)})^\top \label{eq:h2tightset_claim}\\
    = (q^s-1)(q^{2d-2-s}-1)\frac{q^{2d-2}-1}{q^2-1}.\notag
  \end{align}
  for all points $P \in \h(2d-1, q^2)$. We have only four possible choices for $P$: that is, $P \in
  H_s$, $P \in H_s^\perp$, $P \in C(H_s)$, or $P \notin H_s \cup H_s^\perp \cup C(H_s)$.

  If $P \in H_s$, then
  \begin{align*}
    |P^\sim \cap H_s| &= q^2 |\h(s-2, q^2)| = q^2 \frac{(q^{s-1}+1)(q^{s-2}-1)}{q^2-1},\\
    |P^\sim \cap H_s^\perp| &= |\h(2d-2-s, q^2)| = \frac{(q^{2d-1-s}+1)(q^{2d-2-s}-1)}{q^2-1},\\
    |P^\sim \cap C(H_s)| &= (q^2-1)|P^\sim \cap H_s^\perp| (1+|P^\sim \cap H_s|) \\
    &= (q^{2d-1-s}+1)(q^{2d-2-s}-1) \left(1 + q^2 \frac{(q^{s-1}+1)(q^{s-2}-1)}{q^2-1}\right).
  \end{align*}
  Hence,
  \begin{align*}
    &((1-q^{2d-2}) \chi_{\{P\}} + \chi_{P^\sim}) ((1-q^{2d-2-s}) \chi_{H_s} + (1-q^{s}) \chi_{H_s^\perp} + \chi_{C(H_s)})^\top\\
    &= (1-q^{2d-2}) (1-q^{2d-2-s}) + (1-q^{2d-2-s}) |P^\sim \cap H_s| + (1-q^{s}) |P^\sim \cap H_s^\perp| + |P^\sim \cap C(H_s)|\\
    &= (q^s-1)(q^{2d-2-s}-1)\frac{q^{2d-2}-1}{q^2-1}.
  \end{align*}
  If $P \in H_s^\perp$, then the intersection numbers are the same as in the case $P \in H_s$ only
  with $(H_s, s)$ replaced by $(H_s^\perp, 2d-2-s)$. If $P \in C(H_s)$, then
  \begin{align*}
    |P^\sim \cap H_s| &= 1 + q^2 |\h(s-2, q^2)| = 1 + q^2 \frac{(q^{s-1}+1)(q^{s-2}-1)}{q^2-1},\\
    |P^\sim \cap H_s^\perp| &= 1 + q^2 |\h(2d-4-s, q^2)| = 1 + q^2 \frac{(q^{2d-3-s}+1)(q^{2d-4-s}-1)}{q^2-1},\\
    |P^\sim \cap C(H_s)|
    &= (|P^\sim \cap H_s| \cdot |P^\sim \cap H_s^\perp|-1) + (|\h(s, q^2)| - |P^\sim \cap H_s|)(|\h(2d-s-2, q^2)| - |P^\sim \cap H_s^\perp|).
  \end{align*}
  This shows \eqref{eq:h2tightset_claim} in this case. If $P \in C(H_s)$, then $|P^\sim \cap H_s| =
  |\h(s-1, q^2)|$, $|P^\sim \cap H_s^\perp| = |\h(2d-s-3, q^2)|$, and $|P^\sim \cap C(H_s)| =
  |\h(s-1, q^2)| \cdot |\h(2d-s-3, q^2)|$. Again, this affirms \eqref{eq:h2tightset_claim}.
\end{proof}

\begin{corollary}\label{corollary:hermitian_ovoid_h2bound}
  Let $d > 2$. Let $H_s$ be a subgeometry of $\h(2d-1, q^2)$ isomorphic to $\h(s, q^2)$, $1 < s <
  2d-2$, $s$ even. Let $\cO$ be an ovoid of $\h(2d-1, q^2)$. Then
 \begin{align*}
    |H_s \cap \cO| \le q^{s+1} - q^{s} + q^{2s-2d+2} + 1.
 \end{align*}
\end{corollary}
\begin{proof}
  Let $\cO$ be an ovoid. We may assume that $H_s \cap \cO$ is non-empty. 
  Let $\chi_T$ be the weighted $(q^s-1)(q^{2d-2-s}-1)$-tight set of Proposition \ref{propn:h2tightset}.
  By Lemma \ref{lemma:intersection_tight_set_ovoid}, $  \chi_{\cO} \chi_T^\top = (q^s-1)(q^{2d-2-s}-1)$,
  and since $H_s^\perp$ is empty, 
  $
   \chi_{\cO} \chi_T^\top = (1-q^{2d-2-s}) |H_s \cap \cO| + |C(H_s) \cap H_s|.
  $
    Hence $|H_s \cap \cO| + |C(H_s) \cap \cO| = (q^s-1)(q^{2d-2-s}-1) + q^{2d-2-s} |H_s \cap \cO|$
  and
  \[
  (q^s-1)(q^{2d-2-s}-1) + q^{2d-2-s} |H_s \cap \cO| \le |\cO| = q^{2d-1}.
  \]
\end{proof}

Using Andreas Klein's arguments together with Corollary \ref{corollary:hermitian_ovoid_h2bound} for
$s=2$, one obtains:

\begin{theorem}\label{theorem:improvement_klein_hermitian}
  The polar space $\h(2d-1, q^2)$, with $d > q^3-q^2+2$, does not possess an ovoid.
\end{theorem}

The same arguments yield a similar bound for $\q^+(2d-1, q)$, which was not considered in \cite{Klein2004}.

\begin{theorem}\label{theorem:improvement_klein_hyperbolic}
  The polar space $\q^+(2d-1, q)$, with $d > q^2-q+3$, does not possess an ovoid.
\end{theorem}

The best known bounds on ovoids of $\h(2d-1, q^2)$ or $\q^+(2d-1, q)$ are due to Blokhuis and
Moorhouse \cite{Blokhuis1995}. In contrast to the results here their proof is purely algebraic and
gives no information on the hypothetical structure of an ovoid. Similar arguments also give the
following existence conditions on ovoids of parabolic quadrics, but in this case better geometric
results are known \cite{Ball2006, GM:1997, Thas:1981kx}.

\begin{theorem}\label{theorem:improvement_klein_parabolic2}
  The polar space $\q(2d, q)$, with $d > \frac{q^2+3}{2}$, does not possess an ovoid.
\end{theorem}

\appendix

\section{The weighted tight set used for $\h(5, 4)$}

We use the Hermitian form $x_1^3+x_2^3+x_3^3+x_4^3+x_5^3+x_6^3$.
Let $\{ 0, 1, a, a^2\}$ be the elements of $\gf(4)$.


Let $U_1$
be the setwise stabiliser of the points spanned by the following eight vectors:
\begin{align*}
&(0,1,1,0,a,a),&(1,a^2,1,0,0,a),\\
&(1,0,0,1,a^2,a),&(0,0,1,a,a,a^2),\\
&(1,a^2,1,0,0,1),&(1,a^2,a^2,a,a,1),\\
&(1,1,0,1,0,1),&(1,1,a,a,1,a^2).
\end{align*}


The group $U$ described in Subsection \ref{sec:h54} is isomorphic to $U_1 \cap \mathsf{PGU}(6,2)$.
The group has $34$ point orbits, here given by their representatives.
\begin{align*}
  &(1,1,0,0,0,0) &&(0,0,0,0,1,1)\\
  &(0,0,0,0,1,a) &&(0,0,0,0,1,a^2)\\
  &(0,0,0,1,0,1) &&(0,0,0,1,0,a^2) \\
  &(0,0,0,1,1,0) &&(0,0,0,1,a,0)\\
  &(0,0,0,1,a^2,0) &&(0,0,1,0,0,1)\\
  &(0,0,1,0,0,a^2) &&(0,0,1,0,a^2,0)\\
  &(0,0,1,1,1,1) &&(0,0,1,1,1,a)\\
  &(0,0,1,1,1,a^2) &&(0,0,1,a,a,1)\\
  &(0,0,1,a,a,a) &&(0,0,1,a,a,a^2)\\
  &(0,0,1,a,a^2,1) &&(0,0,1,a,a^2,a)\\
  &(0,0,1,a^2,0,0) &&(0,0,1,a^2,a^2,1)\\
  &(0,0,1,a^2,a^2,a) &&(0,1,0,1,a,1)\\
  &(0,1,0,a,0,0) &&(0,1,0,a,a,1)\\
  &(0,1,0,a,a,a^2) &&(0,1,0,a^2,1,a^2)\\
  &(0,1,a,1,a^2,0) &&(0,1,a,a,1,0)\\
  &(1,0,0,1,a^2,a) &&(1,0,a,0,1,a^2)\\
  &(1,0,a^2,a,0,a^2) &&(1,a^2,1,0,0,a)
\end{align*}
%
%
%
%

Put the following weights on these orbits (with the same order) to obtain a weighted $36$-tight set.
\begin{align*}
 (5, 0, -2, 0, -7, 3, 16, 0, 2, 0, 10, 0, 0, 0, 24, 12, 0, 36, 0, 0,\\ -9, 6, -6, -6, 0, 0, 2, -8, 6, 24, -12, 0, 0, 0)
\end{align*}
In the notation of Subsection \ref{sec:h54}, $\ell$ is the 30th orbit, $P$ is the
34th orbit, and the $9$ remaining totally isotropic points of $\langle P, \ell \rangle$
are in the 14th, 31st and 33rd orbit. The assumption says that there are $3$ non-collinear points of the ovoid in
the 14th orbit. All orbits except $2, 4, 8, 10, 12, 13, 14, 15, 16, 18, 19, 25, 26, 32$
contain no point of the ovoid, since their points are in the perp of one point
of the 14th orbit. 
The orbits $16, 18, 25, 32$
are empty, since they are subsets of $W_5$ and all $3$ points of the ovoid in $W_5$
are in the 14th orbit. Let $O_i$ denote the $i$-th point orbit. Hence,
\begin{align*}
  24 |O_{15} \cap \cO| = 36.
\end{align*}
This is a contradiction, since $|O_{15} \cap \cO|$ is an integer.


\begin{thebibliography}{10}

\bibitem{Ball2006}
S.~Ball, P.~Govaerts, and L.~Storme.
\newblock On ovoids of parabolic quadrics.
\newblock {\em Des. Codes Cryptogr.}, 38(1):131--145, 2006.

\bibitem{Bamberg2012}
J.~Bamberg, A.~Devillers, and J.~Schillewaert.
\newblock Weighted intriguing sets of finite generalised quadrangles.
\newblock {\em J. Algebraic Combin.}, 36(1):149--173, 2012.

\bibitem{Bamberg2007}
J.~Bamberg, S.~Kelly, M.~Law, and T.~Penttila.
\newblock Tight sets and {$m$}-ovoids of finite polar spaces.
\newblock {\em J. Combin. Theory Ser. A}, 114(7):1293--1314, 2007.

\bibitem{Barlotti:1955yq}
A.~Barlotti.
\newblock Un'estensione del teorema di {S}egre-{K}ustaanheimo.
\newblock {\em Boll. Un. Mat. Ital. (3)}, 10:498--506, 1955.

\bibitem{Blokhuis1995}
A.~Blokhuis and G.~E. Moorhouse.
\newblock Some {$p$}-ranks related to orthogonal spaces.
\newblock {\em J. Algebraic Combin.}, 4(4):295--316, 1995.

\bibitem{Brouwer1989}
A.~E. Brouwer, A.~M. Cohen, and A.~Neumaier.
\newblock {\em Distance-regular graphs}, volume~18 of {\em Ergebnisse der
  Mathematik und ihrer Grenzgebiete (3) [Results in Mathematics and Related
  Areas (3)]}.
\newblock Springer-Verlag, Berlin, 1989.

\bibitem{Buekenhout:1974kx}
F.~Buekenhout and E.~Shult.
\newblock On the foundations of polar geometry.
\newblock {\em Geometriae Dedicata}, 3:155--170, 1974.

\bibitem{Cooperstein:1995aa}
B.~N. Cooperstein.
\newblock A note on tensor products of polar spaces over finite fields.
\newblock {\em Bull. Belg. Math. Soc. Simon Stevin}, 2(3):253--257, 1995.

\bibitem{DBKMnova}
J.~De~Beule, A.~Klein, and K.~Metsch.
\newblock Substructures of finite classical polar spaces.
\newblock In {\em Current research topics in Galois geometry}, chapter~2, pages
  33--59. Nova Sci. Publ., New York, 2012.

\bibitem{DeBeule2005}
J.~De~Beule and K.~Metsch.
\newblock The {H}ermitian variety {$H(5,4)$} has no ovoid.
\newblock {\em Bull. Belg. Math. Soc. Simon Stevin}, 12(5):727--733, 2005.

\bibitem{Delsarte1973}
P.~Delsarte.
\newblock An algebraic approach to the association schemes of coding theory.
\newblock {\em Philips Res. Rep. Suppl.}, (10):vi+97, 1973.

\bibitem{Delsarte1977}
P.~Delsarte.
\newblock Pairs of vectors in the space of an association scheme.
\newblock {\em Philips Res. Rep.}, 32(5-6):373--411, 1977.

\bibitem{Dembowski:1997fj}
P.~Dembowski.
\newblock {\em Finite geometries}.
\newblock Classics in Mathematics. Springer-Verlag, Berlin, 1997.
\newblock Reprint of the 1968 original.

\bibitem{Eisfeld1998a}
J.~Eisfeld.
\newblock On the common nature of spreads and pencils in ${PG}(d, q)$.
\newblock {\em Discrete Mathematics}, 189(1-3):95--104, 1998.

\bibitem{GM:1997}
A.~Gunawardena and G.~E. Moorhouse.
\newblock The non-existence of ovoids in {$O_9(q)$}.
\newblock {\em European J. Combin.}, 18(2):171--173, 1997.

\bibitem{Hirschfeld1991}
J.~W.~P. Hirschfeld and J.~A. Thas.
\newblock {\em General {G}alois geometries}.
\newblock Oxford Mathematical Monographs. The Clarendon Press Oxford University
  Press, New York, 1991.
\newblock Oxford Science Publications.

\bibitem{Kantor1982}
W.~M. Kantor.
\newblock Ovoids and translation planes.
\newblock {\em Canad. J. Math.}, 34(5):1195--1207, 1982.

\bibitem{Klein2004}
A.~Klein.
\newblock Partial ovoids in classical finite polar spaces.
\newblock {\em Des. Codes Cryptogr.}, 31(3):221--226, 2004.

\bibitem{Moorhouse1996}
G.~E. Moorhouse.
\newblock Some {$p$}-ranks related to {H}ermitian varieties.
\newblock {\em J. Statist. Plann. Inference}, 56(2):229--241, 1996.
\newblock Special issue on orthogonal arrays and affine designs, Part II.

\bibitem{Segre:1959rz}
B.~Segre.
\newblock Le geometrie di {G}alois. {A}rchi ed ovali; calotte ed ovaloidi.
\newblock {\em Confer. Sem. Mat. Univ. Bari}, 43-44:31 pp. (1959), 1959.

\bibitem{Segre:1959}
B.~Segre.
\newblock On complete caps and ovaloids in three-dimensional {G}alois spaces of
  characteristic two.
\newblock {\em Acta Arith.}, 5:315--332 (1959), 1959.

\bibitem{Shult:2005yu}
E.~E. Shult.
\newblock Problems by the wayside.
\newblock {\em Discrete Math.}, 294(1-2):175--201, 2005.

\bibitem{Thas:1981kx}
J.~A. Thas.
\newblock Ovoids and spreads of finite classical polar spaces.
\newblock {\em Geom. Dedicata}, 10(1-4):135--143, 1981.

\bibitem{Thas:2001ly}
J.~A. Thas.
\newblock Ovoids, spreads and {$m$}-systems of finite classical polar spaces.
\newblock In {\em Surveys in combinatorics, 2001 ({S}ussex)}, volume 288 of
  {\em London Math. Soc. Lecture Note Ser.}, pages 241--267. Cambridge Univ.
  Press, Cambridge, 2001.

\bibitem{Tits:1962vn}
J.~Tits.
\newblock Ovo\"\i des \`a translations.
\newblock {\em Rend. Mat. e Appl. (5)}, 21:37--59, 1962.

\bibitem{Tits:1974ys}
J.~Tits.
\newblock {\em Buildings of spherical type and finite {BN}-pairs}.
\newblock Lecture Notes in Mathematics, Vol. 386. Springer-Verlag, Berlin,
  1974.

\end{thebibliography}

\end{document}